\documentclass[11pt]{amsart}

\usepackage{amsmath,amssymb,booktabs,array}
\usepackage[T1]{fontenc}
\usepackage{lmodern}
\usepackage{microtype}
\usepackage{xcolor}
\usepackage{hyperref}

\hypersetup{colorlinks=true,linkcolor=blue!55!black,urlcolor=blue!55!black}

\newtheorem{theorem}{Theorem}[section]
\newtheorem{proposition}[theorem]{Proposition}

\newcommand{\N}{\mathbb N}
\newcommand{\ind}{\mathbf 1}
\newcommand{\rankapp}{z}

\title[The Fibonacci numbers are not an AU set]
{The Fibonacci numbers are not an additive uniqueness set for multiplicative functions}
\author{Poo-Sung Park}
\address{Department of Mathematics Education, Kyungnam University,
Changwon 51767, Republic of Korea}
\email{pspark@kyungnam.ac.kr}
\date{September 8, 2026}

\begin{document}
\begin{abstract}
Let $(F_n)_{n\geq 0}$ be the Fibonacci sequence.  We show that a
multiplicative function $f$ satisfying
\[
 f(F_n+F_m)=f(F_n)+f(F_m)\qquad(n,m\geq 1)
\]
need not be the identity function, even when $f$ takes positive integer
values.  This answers negatively a question posed by Spiro in 1992.
The smallest example presented here is obtained from
$F_{31}=557\cdot 2417$.  We prove the required divisibility equivalence,
formulate an abstract prime-signature construction, and give a practical
criterion producing further examples.  We also describe the AI-assisted
search that led to the construction and provide a reproducible certificate
checker.
\end{abstract}
\maketitle
\raggedbottom

\section{Introduction}

Let $\mathcal M$ denote the class of multiplicative functions
$f:\mathbb N\to\mathbb C$.  A set $E\subseteq\mathbb N$ is called an
\emph{additive uniqueness set}, or an \emph{AU set}, for $\mathcal M$ if the
condition
\[
 f(a+b)=f(a)+f(b)\qquad(a,b\in E)
\]
determines $f\in\mathcal M$ uniquely.  Since the identity function always
satisfies this condition, uniqueness means that $f(n)=n$ for every
$n\in\mathbb N$.

Spiro introduced this viewpoint in 1992 in her foundational paper
\cite{Spiro1992}.  Among other results, she proved that the set of primes is
an AU set for multiplicative functions that do not vanish at every prime.
In the same paper she singled out the set of positive Fibonacci numbers and
asked whether it is an AU set.  The question is especially natural because
the Fibonacci recurrence immediately makes the condition very rigid on the
sequence itself.  Indeed, if
\begin{equation}\label{eq:original-condition}
 f(F_n+F_m)=f(F_n)+f(F_m)\qquad(n,m\geq1),
\end{equation}
then elementary arguments give $f(F_n)=F_n$ for all $n$.  Nevertheless, this
does not by itself determine the values of $f$ away from the Fibonacci
sequence.

Despite substantial progress on AU sets formed by primes, polygonal numbers,
and related sequences, Spiro's Fibonacci question remained open for more than
three decades; it was still listed as an open problem in a 2025 survey
\cite[Section~4.1]{Park2025}.  The same survey records the partial result that
$f(2^a)=2^a$ for every $a\geq0$.  One may also readily deduce
$f(L_n)=L_n$ for the Lucas numbers.  These facts suggested that the identity
function might be forced, but they do not rule out a deformation supported
on primes that are indistinguishable by Fibonacci sums.

The purpose of this note is to answer Spiro's question negatively.  We give
an explicit positive integer-valued multiplicative function that satisfies
\eqref{eq:original-condition} but is not the identity.  To the best of the
author's knowledge, this is the first counterexample to the Fibonacci AU
problem.  The key observation is that the two primes
\[
 557\quad\text{and}\quad2417,
\]
which are the prime factors of $F_{31}$, have exactly the same divisibility
pattern on every Fibonacci number and every sum of two Fibonacci numbers.
This permits their multiplicative contributions to be exchanged without
altering any value tested by \eqref{eq:original-condition}.

The paper is organized as follows.  Section~\ref{sec:explicit} proves the
explicit counterexample.  Section~\ref{sec:signature} isolates the general
prime-signature mechanism, and Section~\ref{sec:criterion} gives a finite
modular criterion for producing such prime pairs.  Further examples are
listed in Section~\ref{sec:further}, and Section~\ref{sec:ai} records the
AI-assisted discovery and exact computational verification.

\section{The explicit counterexample}\label{sec:explicit}

We use the normalization
\[
 F_0=0,\qquad F_1=1,\qquad F_{n+2}=F_{n+1}+F_n.
\]
A function $f:\N\to\mathbb C$ is called multiplicative if $f(1)=1$ and
$f(ab)=f(a)f(b)$ whenever $\gcd(a,b)=1$.  The condition under consideration
is
\begin{equation}\label{eq:additivity}
 f(F_n+F_m)=f(F_n)+f(F_m)\qquad(n,m\geq1).
\end{equation}

Set
\[
 p=557,\qquad q=2417;\qquad F_{31}=1346269=pq.
\]
Define $f:\N\to\N$ by
\begin{equation}\label{eq:counterexample}
 f(N)=N\left(\frac qp\right)^{\ind_{p\mid N}-\ind_{q\mid N}},
\end{equation}
where $\ind_{d\mid N}$ is $1$ if $d\mid N$ and $0$ otherwise.  Equivalently,
\[
f(N)=
\begin{cases}
 (q/p)N,&p\mid N,\ q\nmid N,\\[1mm]
 (p/q)N,&q\mid N,\ p\nmid N,\\[1mm]
 N,&\text{otherwise}.
\end{cases}
\]
The quotients in the first two cases are positive integers.  Moreover, for
coprime $A,B$ one has
\[
 \ind_{p\mid AB}=\ind_{p\mid A}+\ind_{p\mid B},
 \qquad
 \ind_{q\mid AB}=\ind_{q\mid A}+\ind_{q\mid B}.
\]
Consequently, $f$ is multiplicative.  It is not the identity, since
\[
 f(557)=2417,\qquad f(2417)=557.
\]

The entire construction rests on the following fact.

\begin{theorem}\label{thm:key31}
For every $n,m\geq1$,
\[
 557\mid F_n+F_m\quad\Longleftrightarrow\quad
 2417\mid F_n+F_m.
\]
\end{theorem}

We prove a more general criterion in Section~\ref{sec:criterion}.  Applying
the theorem first to a single Fibonacci number and then to a sum gives
\[
 f(F_n)=F_n,
 \qquad f(F_n+F_m)=F_n+F_m.
\]
Thus \eqref{eq:additivity} holds, while $f$ is not the identity.  Notice also
that the same example satisfies $f(L_n)=L_n$ for every Lucas number $L_n$ and
$f(2^a)=2^a$ for every $a\geq0$.  For $n\geq2$ this follows from
$L_n=F_{n-1}+F_{n+1}$; the remaining initial case is immediate.

\section{A general prime-signature construction}\label{sec:signature}

The mechanism is not peculiar to Fibonacci numbers.  Let $\mathcal A$ be a
set of positive integers and associate with each prime $\ell$ its divisibility
signature on $\mathcal A$,
\[
 \sigma_{\ell,\mathcal A}(A)=\ind_{\ell\mid A}\qquad(A\in\mathcal A).
\]

\begin{proposition}[signature switching]\label{prop:signature}
Suppose that distinct primes $p$ and $q$ have the same signature on
$\mathcal A$.  Then, for every $t\in\mathbb C^\times$,
\[
 f_t(N)=N t^{\ind_{p\mid N}-\ind_{q\mid N}}
\]
is multiplicative and fixes every member of $\mathcal A$.  If $t\ne1$, it is
not the identity.  In particular, the choice $t=q/p$ gives a positive
integer-valued function.
\end{proposition}

\begin{proof}
Multiplicativity follows from additivity of the two divisibility indicators
on coprime products.  Equality of signatures makes the exponent zero on every
$A\in\mathcal A$.  For $t=q/p$, the only potentially nontrivial values are
$(q/p)N$ with $p\mid N$ and $(p/q)N$ with $q\mid N$, so the values are positive
integers.
\end{proof}

For the present problem take
\[
 \mathcal A=\{F_n:n\geq1\}\cup
 \{F_n+F_m:n,m\geq1\}.
\]
Thus the search for counterexamples becomes the search for two primes with
identical signatures on this set.  Proposition~\ref{prop:signature} also shows
that one successful pair produces infinitely many complex-valued
counterexamples.  Several disjoint successful pairs can be switched
independently.

\section{A Fibonacci criterion}\label{sec:criterion}

For a prime $\ell$, let $\rankapp(\ell)$ denote its \emph{rank of apparition},
that is, the least positive $r$ for which $\ell\mid F_r$.

\begin{theorem}\label{thm:criterion}
Let $r\geq5$ be odd, and let $p$ and $q$ be distinct odd primes.  Suppose that
\begin{enumerate}
\item $\rankapp(p)=\rankapp(q)=r$;
\item for each $\ell\in\{p,q\}$, the residues
\begin{equation}\label{eq:fourth-residues}
 F_i^4\pmod\ell,\qquad 2\leq i\leq (r-1)/2,
\end{equation}
are pairwise distinct.
\end{enumerate}
Then, for all $n,m\geq1$,
\[
 p\mid F_n+F_m\quad\Longleftrightarrow\quad q\mid F_n+F_m.
\]
Consequently, \eqref{eq:counterexample}, with this $p,q$, is a positive
integer-valued counterexample to \eqref{eq:additivity}.
\end{theorem}

\begin{proof}
Fix $\ell\in\{p,q\}$ and put
\[
 c_\ell\equiv F_{r+1}\pmod\ell.
\]
Since $F_r\equiv0\pmod\ell$, the addition formula and Cassini's identity give
\begin{equation}\label{eq:shift}
 F_{n+r}\equiv c_\ell F_n\pmod\ell,
 \qquad c_\ell^2\equiv-1\pmod\ell.
\end{equation}
In particular, $c_\ell$ has order four.  A second application of the addition
formula gives
\begin{equation}\label{eq:reflection}
 F_{r-i}\equiv(-1)^{i+1}c_\ell F_i\pmod\ell.
\end{equation}

If $r\nmid n$, reduce $n$ modulo $r$ and use
\eqref{eq:shift}--\eqref{eq:reflection}.  Since $F_1=F_2=1$, this yields a
representation
\begin{equation}\label{eq:normal-form}
 F_n\equiv c_\ell^{e(n)}F_{i(n)}\pmod\ell,
 \qquad
 2\leq i(n)\leq\frac{r-1}{2},\quad e(n)\in\mathbb Z/4\mathbb Z.
\end{equation}
The pair $i(n),e(n)$ can be chosen solely from the residue class of $n$ modulo
$4r$, and hence independently of $\ell$.

Assume first that $r\nmid nm$.  By \eqref{eq:normal-form},
\[
 F_n+F_m\equiv0\pmod\ell
\]
is equivalent to
\[
 c_\ell^{e(n)}F_{i(n)}=-c_\ell^{e(m)}F_{i(m)}.
\]
Taking fourth powers and using hypothesis~(2) forces $i(n)=i(m)$.  Cancellation
followed by $c_\ell^2=-1$ then gives
\[
 e(n)-e(m)\equiv2\pmod4.
\]
This condition is independent of $\ell$.  Finally,
$F_n\equiv0\pmod\ell$ if and only if $r\mid n$, by hypothesis~(1).  If exactly
one of $n,m$ is divisible by $r$, the sum is nonzero modulo $\ell$; if both
are divisible by $r$, it is zero.  The result follows in every case.
\end{proof}

For $r=31$, the certificates required in Theorem~\ref{thm:criterion} are
\begin{align*}
 \rankapp(557)=\rankapp(2417)=31,
 \qquad &F_{32}\equiv439\pmod{557},\\
 &F_{32}\equiv592\pmod{2417}.
\end{align*}
Moreover, $439^2\equiv-1\pmod{557}$ and
$592^2\equiv-1\pmod{2417}$.  The fourth-power
residues are displayed below; each column has no repetition.
{\scriptsize\setlength{\arraycolsep}{2pt}
\[
\begin{array}{c|rrrrrrrrrrrrrr}
i&2&3&4&5&6&7&8&9&10&11&12&13&14&15\\ \hline
F_i^4\bmod557
&1&16&81&68&197&154&88&93&229&90&533&203&481&19\\
F_i^4\bmod2417
&1&16&81&625&1679&1974&1121&2152&2280&1755&2230&887&2312&1839
\end{array}
\]
}
This proves Theorem~\ref{thm:key31}.

\section{Further counterexamples}\label{sec:further}

The same test gives the following additional pairs.  In every row the two
displayed primes have rank of apparition $r$, their Pisano period is $4r$, and
each satisfies the distinctness condition \eqref{eq:fourth-residues}.
\begin{center}
\begin{tabular}{r@{\qquad}r@{\qquad}r@{\qquad}r}
\toprule
$r$ & $p$ & $q$ & common Pisano period\\
\midrule
31&557&2417&124\\
41&2789&59369&164\\
55&661&474541&220\\
61&4513&555003497&244\\
67&116849&1429913&268\\
73&9375829&86020717&292\\
\bottomrule
\end{tabular}
\end{center}
For example,
\[
 F_{41}=165580141=2789\cdot59369.
\]
Hence replacing $557,2417$ in \eqref{eq:counterexample} by $2789,59369$
produces another counterexample, with
\[
 f(2789)=59369,\qquad f(59369)=2789.
\]
Likewise, $F_{55}=5\cdot89\cdot661\cdot474541$, and the pair
$661,474541$ may be switched.

The condition that two primes divide the same Fibonacci number is not enough.
For instance,
\[
 F_{19}=37\cdot113,\qquad F_9+F_4=34+3=37.
\]
Thus $37\mid F_9+F_4$, whereas $113\nmid F_9+F_4$, so these two primes have
different signatures and cannot be switched by Proposition~\ref{prop:signature}.

\section{AI-assisted discovery and verification}\label{sec:ai}

The initial example was found in September 2026 in an interactive ChatGPT
session using \textbf{OpenAI GPT-6 Astra}.  The model was used in two distinct
ways.
\begin{enumerate}
\item \emph{Structural step.}  It proposed replacing the value at one prime by
the other prime and recognized that multiplicativity is preserved by the
indicator formula in Proposition~\ref{prop:signature}.  This reduced the
problem to finding a collision of prime divisibility signatures.
\item \emph{Computational step.}  It designed and ran an exact-integer search:
factor $F_r$ for odd $r$, retain prime divisors having rank of apparition $r$,
and compare their divisibility patterns on Fibonacci sums.  The pair
$(557,2417)$ arising from $F_{31}$ was the first successful pair retained in
this search.  The modular pattern was then compressed into the fourth-power
criterion of Theorem~\ref{thm:criterion}, which supplied a short proof and
made the further rows above easy to verify.
\end{enumerate}
The language model's output was therefore a discovery aid, not a substitute
for verification.  All claims in this note reduce to exact factorizations and
finite modular computations, and the argument proving that these computations
cover every $n,m$ is given in Theorem~\ref{thm:criterion}.

For reproducibility, the following Python fragment checks the certificate for
any proposed triple $(r,p,q)$.  It uses only exact modular arithmetic; primality
and the displayed factorization should be checked separately when new pairs
are sought.
\begin{verbatim}
def fibmod(n, modulus):
    a, b = 0, 1
    for _ in range(n):
        a, b = b, (a + b) % modulus
    return a

def rank_of_apparition(prime, bound):
    a, b = 0, 1
    for n in range(1, bound + 1):
        a, b = b, (a + b) % prime
        if a == 0:
            return n
    return None

def check_certificate(r, p, q):
    assert r % 2 == 1 and r >= 5
    for ell in (p, q):
        assert rank_of_apparition(ell, r) == r
        c = fibmod(r + 1, ell)
        assert c * c % ell == ell - 1
        residues = [pow(fibmod(i, ell), 4, ell)
                    for i in range(2, (r - 1)//2 + 1)]
        assert len(residues) == len(set(residues))
    return True

examples = [(31, 557, 2417),
            (41, 2789, 59369),
            (55, 661, 474541),
            (61, 4513, 555003497),
            (67, 116849, 1429913),
            (73, 9375829, 86020717)]
for triple in examples:
    assert check_certificate(*triple)
\end{verbatim}

\section{Concluding remarks}

The equation \eqref{eq:additivity} determines $f$ on every Fibonacci number
and on every sum of two Fibonacci numbers, but this set does not distinguish
all prime divisibility signatures.  A collision between two such signatures
leaves room to move multiplicative mass from one prime to the other without
changing any tested value.  This is the general principle behind the examples.

It remains natural to ask whether infinitely many distinct prime pairs satisfy
the criterion of Theorem~\ref{thm:criterion}, or even the weaker identical-
signature condition of Proposition~\ref{prop:signature}.  No infinitude claim
is needed for the counterexamples above.

\end{document}